\documentclass[12pt]{article}
\usepackage{amsmath,amsfonts,epsfig}
\title{\bf On normal subgroups of an amalgamated product of groups with applications to knot theory}
\author{John G. Ratcliffe\\ 
Mathematics Department, Vanderbilt University, \\
Nashville TN 37240, USA \\ \\
In memory of John Stallings}
\newtheorem{theorem}{Theorem}[section]

\newtheorem{corollary}[theorem]{Corollary}

\newenvironment{example}{\vspace{0in}{\noindent\bf Example:\ }}{}
\newenvironment{remark}{\vspace{0in}{\noindent\bf Remark: }}{}
\newenvironment{proof}{{\bf Proof:\ }}{\hfill$\square$\vspace{.2in}}
\newcommand{\integers}{{\mathbb Z}}

\date{}
\begin{document}
\maketitle

\noindent {\bf Abstract:} In this paper, useful necessary and sufficient conditions for 
a normal subgroup of an amalgamated product of groups to be finitely generated are given. 
These conditions are applied together with Stallings' fibering theorem to give a simple proof 
that an irreducible multilink in a homology 3-sphere fibers if and only if each of its multilink splice components fibers. 

\section{Introduction} 

In a recent paper \cite{H-M-S}, M. Hirasawa, K. Murasugi, and D. Silver gave useful necessary and sufficient conditions for a satellite knot in $S^3$ to fiber.  Their proof used J. Stallings' fibering theorem \cite{S} 
together with a group theoretical argument about an amalgamated product of groups. 
Hirasawa et al.  remarked in their paper that their theorem follows from Theorem 4.2 of the monograph 
\cite{E-N} of D. Eisenbud and W. Neumann, which states that a multilink in a homology 3-sphere is fibered if and only if it is irreducible and each of its multilink splice components is fibered.  Eisenbud and Neumann 
used the theory of foliations to prove their fibering theorem, and they remarked that their fibering theorem could also be proved using a group theoretic argument and Stallings' fibering theorem. 

The original motivation for this paper was to give a group theoretical argument sufficient to prove Theorem 4.2 
of \cite{E-N} using Stallings' fibering theorem.  In the process, we discovered some useful necessary and sufficient conditions for a normal subgroup of an amalgamated product of groups to be finitely generated.  
These conditions, listed in Theorem 2.1, and their HNN versions, listed in Theorem 2.2,  are the main results of this paper. 
Our main results can be applied to simplify the proofs of theorems of Eisenbud and Neumann \cite{E-N}, 
Moon \cite{M}, \cite{M2}, and Sykiotis \cite{Sykiotis} as explained below. 
Some related theorems concerning free normal subgroups, 
one-relator groups, and cable knots are also proved, which may be of independent interest. 

\section{Finitely generated normal subgroups} 

We begin by considering conditions for a normal subgroup of an amalgamated product of groups 
to be finitely generated.  An amalgamated product of groups $A\ast_C B$ is said to be {\it nontrivial} 
if and only if $A\neq C\neq B$. 

\begin{theorem}  
Let $N$ be a normal subgroup of a group $G$ with a nontrivial amalgamated product decomposition 
$G=A\ast_CB$.  
If $N$ is finitely generated and $N\not\subseteq C$, then $NC$ is of finite index in $G$, 
and if  moreover $N\cap C$ is finitely generated, then $N\cap A$ and $N\cap B$ are finitely generated. 
Conversely, if $NC$ is of finite index in $G$, then $N\not\subseteq C$, 
and if moreover $N\cap A$ and $N\cap B$ are finitely generated, then $N$ is finitely generated. 
If $NC$ is of finite index in $G$, 
then $G/N$ is finite if and only if $C/N\cap C$ is finite. 
\end{theorem}
\begin{proof}
By the subgroup theorem of Karrass and Solitar \cite{K-S} for amalgamated products 
or Theorem 3.14 of P. Scott and C.T.C  Wall \cite{S-W},  the group 
$N$ is the  product of a graph $\Gamma$ of subgroups described as follows.  
There are sets of representatives $\{g_\alpha\}, \{g_\beta\}, \{g_\gamma\}$ 
for the $NA$, $NB$, $NC$ cosets, respectively, such that the representative of $NA$, $NB$, and $NC$ is 1. 
Moreover, the vertices of $\Gamma$ are the cosets $g_\alpha NA$ and $g_\beta NB$, and the corresponding groups are $N\cap g_\alpha A g_\alpha^{-1}$ and $N\cap g_\beta Bg_\beta^{-1}$.  The edges of $\Gamma$ are the cosets $g_\gamma NC$, and the corresponding groups are $N\cap g_\gamma Cg_\gamma^{-1}$. 
The two ends of the edge $g_\gamma NC$ are the vertices $g_\gamma NA$ and $g_\gamma NB$. 
If $g_\gamma NA = g_\alpha NA$, then $g_\gamma A = g_\alpha A$ and if $g_\gamma NB = g_\beta NB$, 
then $g_\gamma B =  g_\beta B$. 

Suppose $N$ is finitely generated and $N\not\subseteq C$. 
Then $NC$ is of finite index in $G$ by Theorem 10 of Karrass and Solitar \cite{K-S}.  
For a proof based on a group actions on trees, see Proposition 2.2 of Bridson and Howie \cite{B-H}.  
Hence $NA$ and $NB$ are of finite index in $G$. 
Therefore $\Gamma$ is a finite graph. 

Suppose moreover that $N\cap C$ is finitely generated
As $N\cap g_\gamma C g_\gamma^{-1} = g_\gamma (N\cap C) g_\gamma^{-1}$, 
all the edge groups of $\Gamma$ are finitely generated.  
By Lemma 2 of Cohen \cite{C}, each vertex group of $\Gamma$ is finitely generated, 
and so $N\cap A$ and $N\cap B$ are finitely generated. 

Conversely, suppose that $NC$ is of finite index in $G$. 
As $C$ is of infinite index in $G$, we deduce that $N\not\subseteq C$. 
Suppose moreover that $N\cap A$ and $N\cap B$ are finitely generated. 
Then $\Gamma$ is a finite graph as before. 
As $N\cap g_\alpha A g_\alpha^{-1} = g_\alpha (N\cap A) g_\alpha^{-1}$ and 
$N\cap g_\beta B g_\beta^{-1} = g_\beta (N\cap B) g_\beta^{-1}$, all the vertex groups of $\Gamma$ 
are finitely generated.  
Therefore $N$ is finitely generated by Theorem 5 of \cite{K-S}. 

Finally, suppose $NC$ is of finite index in $G$. 
As $NC/N \cong C/N\cap C$, we have that $G/N$ is finite if and only if $C/N\cap C$ is finite. 
\end{proof}

We next consider an HNN version of Theorem 2.1. 

\begin{theorem}  
Let $N$ be a normal subgroup of a group $G$ with an HNN decomposition 
$G=A\ast_C = \langle A,t : tCt^{-1} = D\rangle$.  
If $N$ is finitely generated and $N\not\subseteq C$, then $NC$ is of finite index in $G$, 
and if  moreover $N\cap C$ is finitely generated, then $N\cap A$ is finitely generated. 
Conversely, if $NC$ is of finite index in $G$, then $N\not\subseteq C$, 
and if moreover $N\cap A$ is finitely generated, then $N$ is finitely generated. 
If $NC$ is of finite index in $G$, 
then $G/N$ is finite if and only if $C/N\cap C$ is finite. 
\end{theorem}
\begin{proof}
By the subgroup theorem of Karrass, Pietrowski, and Solitar \cite{K-P-S} for HNN groups 
or Theorem 3.14 of P. Scott and C.T.C  Wall \cite{S-W}, 
the group $N$ is the product of a graph $\Gamma$ of subgroups described as follows.  
There are sets of representatives $\{g_\alpha\}$ and $\{g_\gamma\}$ 
for the $NA$ and $ND$ cosets, respectively, such that the representative of $NA$ and $ND$ is 1. 
Moreover, the vertices of $\Gamma$ are the cosets $g_\alpha NA$, and the corresponding groups are 
$N\cap g_\alpha A g_\alpha^{-1}$.  The edges of $\Gamma$ are the cosets $g_\gamma ND$, and the corresponding groups are $N\cap g_\gamma Dg_\gamma^{-1}$. 
The two ends of the edge $g_\gamma ND$ are the vertices $g_\gamma NA$ and $g_\gamma tNA$. 
If $g_\gamma NA = g_\alpha NA$, then $g_\gamma A = g_\alpha A$ 
and if $g_\gamma t NA = g_\beta NA$, then $g_\gamma tA=  g_\beta A$. 

Suppose $N$ is finitely generated and $N\not\subseteq C$. 
Then $NA$ is of finite index in $G$ by Theorem 9 of Karrass and Solitar \cite{K-S2}. 
Hence $\Gamma$ has finitely many vertices. 
By Theorem 1 of \cite{K-S}, the graph $\Gamma$ has a spanning tree  ${\rm T }$ such that there are only 
finitely many edges of $\Gamma$ not in ${\rm T }$. 
Hence $\Gamma$ is a finite graph. 
Therefore $ND$ is of finite index in $G$. 
As $ND = t(NC)t^{-1}$, we have that $NC$ is of finite index in $G$. 

Suppose moreover that $N\cap C$ is finitely generated. 
As $N\cap g_\gamma D g_\gamma^{-1} = g_\gamma t (N\cap C)t^{-1} g_\gamma^{-1}$, 
all the edge groups of $\Gamma$ are finitely generated.  
By Lemma 2 of Cohen \cite{C}, each vertex group of $\Gamma$ is finitely generated, 
and so $N\cap A$ is finitely generated. 

Conversely, suppose that $NC$ is of finite index in $G$. 
As $C$ is of infinite index in $G$, we deduce that $N\not\subseteq C$. 
Now $ND$ and $NA$ are of finite index in $G$. 
Hence $\Gamma$ is a finite graph. 
Suppose moreover that $N\cap A$ is finitely generated. 
As $N\cap g_\alpha A g_\alpha^{-1} = g_\alpha (N\cap A) g_\alpha^{-1}$,  
all the vertex groups of $\Gamma$ are finitely generated.  
Therefore $N$ is finitely generated by Theorem 1 of \cite{K-S2}. 

Finally, suppose $NC$ is of finite index in $G$. 
As $NC/N \cong C/N\cap C$, we have that $G/N$ is finite if and only if $C/N\cap C$ is finite. 
\end{proof}

As an application of Theorems 2.1 and 2.2 to group theory, 
we give a short proof of the main result of M. Moon \cite{M}. 

\begin{theorem} {\rm  (M. Moon)} 
Let $G = A\ast _C$ (or $G = A\ast_C B$), where $C$ is a free abelian group of finite rank.  
Suppose that $A$ (and $B$) contain no finitely generated nontrivial normal subgroup of 
infinite index.  If $N$ is a finitely generated normal subgroup of $G$ with $N\not\subseteq C$, 
then $G/N$ is finite or $N$ is free. 
\end{theorem}
\begin{proof}
If $G = A\ast_C B$ is a trivial decomposition, the result is clear, and so in the amalgamated case,  
we may assume that $G = A\ast_C B$ is a nontrivial decomposition. 
We have that $N\cap C$ is finitely generated.   
Hence $N\cap A$ is finitely generated (and $N\cap B$ is finitely generated) by Theorem 2.2 (or 2.1). 
Therefore either $N\cap A =\{1\}$ or $N\cap A$ is of finite index in $A$   
(and either $N\cap B=\{1\}$ or $N\cap B$ is of finite index in $B$) by hypothesis.

Now $NC$ is of finite index in $G$ by Theorem 2.2 (or 2.1). 
Hence $NA$ is of finite index in $G$ (and $NB$ is of finite index in $G$). 
If $A/N\cap A$ is finite, then $A/N\cap A \cong NA/N$ is finite, and so $G/N$ is finite. 

If $N\cap A =\{1\}$ (and $N\cap B=\{1\}$), 
then $N$ is free, since $N$ is the product of a finite graph of subgroups 
all of whose vertex groups are trivial by the proofs of Theorems 2.2 (and 2.1). 
\end{proof}

\begin{remark}  The main result of Moon \cite{M2} (Theorem 2.3) can also be proved 
in the same way.  An application of Theorem 2.1 would also simplify the proof of
Proposition 2.9 of M. Sykiotis \cite{Sykiotis}. 
\end{remark}

\vspace{.2in}
We next consider a corollary of Theorem 2.1 which will be more useful 
for our application to knot theory. 

\begin{corollary} 
Let $G$ be a group with an amalgamated product decomposition $G=A\ast_CB$, 
and  let $\phi: G \to \integers$ be a homomorphism such that ${\rm ker}(\phi |_C)$ is finitely generated 
and not equal to $C$. 
Then ${\rm ker}(\phi)$ is finitely generated if and only if ${\rm ker}(\phi |_A)$ and ${\rm ker}(\phi |_B)$ 
are finitely generated. 
\end{corollary}
\begin{proof}
Let $N = {\rm ker}(\phi)$.  Then $N\cap C = {\rm ker}(\phi |_C)$, $N\cap A = {\rm ker}(\phi |_A)$, 
and $N\cap B = {\rm ker}(\phi |_B)$, and $N\cap C$ is finitely generated. 
If $G=A\ast_CB$ is a trivial decomposition, the result is clear, 
and so we may assume that $G=A\ast_CB$ is a nontrivial decomposition. 
As ${\rm ker}(\phi |_C) \neq C$, the homomorphism $\phi |_C$ is nontrivial. 
Hence $\phi(C)$ is of finite index in $\integers$. 
Therefore $NC$ is of finite index in $G$.  
Hence $N\not\subseteq C$ by Theorem 2.1.  
Therefore $N$ is finitely generated if and only if $N\cap A$ and $N\cap B$ 
are finitely generated by Theorem 2.1. 
\end{proof}

\section{Free normal subgroups} 

We next consider a result about normal subgroups whose proof is similar to the proofs 
of Theorems 2.1 and 2.2 . 

\begin{theorem}  
Let $N$ be a nontrivial normal subgroup of a group $G$ such that $G = A\ast _C$ (or $G = A\ast_C B$), 
and $N\cap C = \{1\}$. 
Then (1) the group $N$ is free if and only if $N\cap A$ is free (and $N\cap B$ is free), 
and (2) the group $N$ is finitely generated if and only if $NC$ is of finite index in $G$ and 
$N\cap A$ is finitely generated (and $N\cap B$ is finitely generated). 
If $N$ is finitely generated and $a= [G:NA]$  (and $b = [G:NB])$ and $c=[G:NC]$, then 
$${\rm rank}(N) = a\,{\rm rank}(N\cap A) + {\big(} b\,{\rm rank}(N\cap B){\big )} + 1 + c\, - a\, - \big(b\big).$$
\end{theorem}
\begin{proof}
By the subgroup theorems of Karrass, Pietrowski, and Solitar \cite{K-S}, \cite{K-P-S}, 
the group $N$ is the product of a graph of subgroups, described in the proofs of Theorems 2.1 and 2.2, 
all of whose edge groups are trivial. 
Hence $N$ is the free product of conjugates of $N\cap A$ (and conjugates of $N\cap B$) and a free group. 
Hence $N$ is free if and only if $N\cap A$ (and $N\cap B$ are free). 

As $N$ is nontrivial and $N\cap C=\{1\}$, we have that $N\not\subseteq C$. 
By Theorems 2.1 and 2.2 we have that $N$ is finitely generated if and only if $NC$ is of finite index in $G$ and $N\cap A$ is finitely generated (and $N\cap B$ is finitely generated). 

Assume that $N$ is finitely generated.  
First assume that $G = A\ast_C$. 
From the proof of Theorem 2.2, we have that $N = P\ast F$, where $P$ is the free product of 
the groups $g_\alpha (N\cap A) g_\alpha^{-1}$, and $\{g_\alpha\}$ is a set of coset representatives 
for $NA$ in $G$, and $F$ is a free group whose rank is the number of edges of the graph $\Gamma$ 
in the complement of a spanning tree ${\rm T}$. 
By Grushko's theorem, we have that 
$${\rm rank}(N) =  {\rm rank}(P) + {\rm rank}(F) = a\,{\rm rank}(N\cap A) + {\rm rank}(F).$$
By collapsing ${\rm T}$, we see that 
$1-{\rm rank}(F) = \chi(\Gamma) = a -c$, and so ${\rm rank}(F) = 1+c - a$. 
The proof for the amalgamated case is similar. 
\end{proof}

We illustrate the use of Theorem 3.1 by proving a theorem about one-relator groups. 
Let $N$ be a normal subgroup of a  one-relator group
$G = \langle x_1,\ldots, x_n ; r\rangle$ such that $G/N$ is infinite cyclic. 
If $n >2$, then $N$ is not finitely generated by Lemma 3 of \cite{B-S}. 
Hence, we will assume $n=2$.  
Then $G_{ab} \cong \integers/m\integers\oplus \integers$ for some integer $m \geq 0$, 
with $m = 0$ if and only if $r$ is in the commutator subgroup 
of the free group $F(x_1,x_2)$. 
Suppose that $m > 0$. Then $m$ is called the {\it torsion number} of $G$,  
and $N$ is the unique normal subgroup of $G$ such that $G/N$ is infinite cyclic, 
and $N$ contains $[G,G]$ as a subgroup of index $m$. 
If $N$ is finitely generated, Moldavanskii \cite{Mold} proved that $N$ is a free group of finite rank. 
The two-generator one-relator groups that are an infinite cyclic extension 
of a free group of rank two were classified up to isomorphism 
by A. Brunner, J. McCool, and A. Pietrowski \cite{B-M-P}.

\begin{theorem}  
Let $G = \langle x,y;r\rangle$ and $H = \langle u, y; s)$, 
suppose $r(x,y) = s(x^e,y)$ for some integer $e > 0$, and the exponent sum of $r$ 
with respect to $y$ is nonzero. 
Then $G$ is an infinite cyclic extension of a free group of finite rank if and only if $H$ is an infinite 
cyclic extension of a free group of finite rank. 

Let $p$ and $q$ be the exponent sums of $r$ with respect to $x$ and $y$, respectively, 
let $m =(p,q)$ be the greatest common divisor of $p$ and $q$, 
and let $a = p/m$ and $b = q/m$.  
If $G$ is an infinite cyclic extension of a group of finite rank $k$ 
and $H$ is an infinite cyclic extension of a group of finite rank $\ell$, then
$$k = 1 + (a,e)(\ell - 1) + |b|(e-1).$$
\end{theorem}
\begin{proof}
We have 
$$G_{ab} \cong \integers\oplus \integers/\langle (p,q)\rangle \cong \integers /m\integers \oplus\integers.$$
Let $N$ be a normal subgroup such that $G/N$ is infinite cyclic. 
Then $N$ is the subgroup of $G$ containing $[G,G]$ that corresponds to the torsion subgroup of $G_{ab}$. 

The torsion number $m$ of $G$ is the greatest common division of $p$ and $q$. 
Let $a = p/m$ and $b=q/m$.  Then $a$ and $b$ are relatively prime. 
Hence there are integers $c$ and $d$ such that $ad-bc=1$. 
Then $(a,b)$ and $(c,d)$ form a basis of $\integers\oplus\integers$ 
and $(a,b) = (p/m,q/m)$ represents an element of order $m$ in $\integers\oplus \integers/\langle (p,q)\rangle$, 
and so $(c,d)$ represents a generator of  $\integers\oplus \integers/\langle (a,b)\rangle \cong \integers$. 
We have $(1,0) =d(a,b)-b(c,d)$. 
As $x$ projects to $(1,0)$ in $\integers\oplus \integers$, 
we have that $x$ projects to $-b$ in $\integers\cong G/N$. 
We have $(0,1) = -c (a,b)+a(c,d)$.
As $y$ projects to $(0,1)$ in $\integers\oplus \integers$, 
we have that $y$ projects to $a$ in $\integers\cong G/N$.

Now as $b\neq 0$, we have that $x$ has infinite order in $G$. 
Similarly, $u$ has infinite order in $H$. 
Hence $G = \langle x\rangle\ast_{x^e=u} H$, 
and so we may identify $H$ with the subgroup $B=\langle x^e,y\rangle$ of $G$. 
Let $\phi: G\to G/N$ be the quotient map, let $A = \langle x\rangle$ and $C = \langle x^e\rangle$. 
Now $\phi(A)$ is infinite cyclic, since $b\neq 0$. 
Hence $N\cap A = {\rm ker}(\phi |_A)=\{1\}$.  
Therefore ${\rm ker}(\phi |_C)=N\cap C=\{1\}$ and $\phi(C)$ is infinite cyclic.  
Hence $N ={\rm ker}(\phi)$ is a free group of finite rank if and only if $N\cap B = {\rm ker}(\phi |_B)$ 
is a free group of finite rank by Corollary 2.4 and Theorem 3.1. 
Moreover $\phi(B)$ is infinite cyclic.

Suppose $N$ has finite rank $k$ and $N\cap B$ has finite rank $\ell$. 
Now $[G:NA]=|b|$, $[G:NB] =(a,be) = (a,e)$, and $[G:NC] = |b|e$. 
By Theorem 3.1, we have 
$$k = (a,e)\ell + 1 + |b|e - |b| -(a,e)=(a,e)(\ell-1) + |b|(e-1).$$

\vspace{-.35in}
\end{proof}

\begin{example}  Let $G = \langle x,y; x^2y^2x^2y^{-1}\rangle$ 
and $H = \langle u, y ;uy^2uy^{-1}\rangle$. 
Then $G$ and $H$ satisfy the hypothesis of Theorem 3.2. 
By applying the automorphism, $u  \mapsto uy$ and $y\mapsto y$, of the free group $F(u,y)$, 
we see that  $\langle u,y;uy^2uy^{-1}\rangle$ is Nielsen equivalent 
to $\langle u,y; u^2y^3\rangle$. 
Now comparing $\langle u,y; u^2y^3\rangle$ with the infinite cyclic group $\langle  v,y; vy^3\rangle$,  
we see that $H$ is an infinite cyclic extension of a free group of rank two by Theorem 3.2.  
Hence $G$ is an infinite cyclic extension of a free group of rank four by Theorem 3.2. 
We will continue with this example in \S 5. 
\end{example}

\section{Application to knot theory} 

We now review some terminology from the first chapter of Eisenbud and Neumann \cite{E-N}. 
A {\it link} ${\bf L} = (\Sigma, K) = (\Sigma, S_1\cup\cdots \cup S_n)$ is a pair consisting of 
an oriented smooth homology 3-sphere $\Sigma$ and a collection of smooth disjoint oriented 
simple closed curves $S_1,\ldots, S_n$ in $\Sigma$, called the {\it components} of ${\bf L}$. 
A {\it knot} is a link with just one component. 
We denote a closed tubular neighborhood of $K$ in $\Sigma$ by $N(K)$. 
Here $N(K) = N(S_1)\cup\cdots\cup N(S_m)$ where $N(S_1), \ldots, N(S_n)$ are 
disjoint, closed, tubular neighborhoods of $S_1,\ldots, S_n$. 
The space $\Sigma_0 = \Sigma - {\rm int}\,N(K)$ is called the {\it link exterior}. 

We denote by $M_i$ and $L_i$ a topologically standard meridian and longitude 
of the link component $S_i$. They are a pair of oriented simple closed curves in $\partial N(S_i)$ 
which are determined up to isotopy by the homology and linking relations: 
$M_i \sim 0, L_i \sim S_i$ in $H_1(N(S_i))$, and $\ell(M_i,S_i) = 1, \ell(L_i,S_i) = 0$, 
where $\ell(\_\,, \_)$ denotes linking number in $\Sigma$. 

Let ${\bf L'}= (\Sigma', K')$ and ${\bf L}''=(\Sigma'',K'')$ be links. 
We can form the {\it disjoint sum} 
$${\bf L}'+ {\bf L}''= (\Sigma' \# \Sigma'', K' \cup K'')$$
by taking the connected sum of $\Sigma'$ and $\Sigma''$ along disks that do not intersect the links. 
If link ${\bf L}$ can be expressed as a nontrivial disjoint sum (that is, neither summand is the empty link 
in $S^3$), then we say that ${\bf L}$ is a {\it reducible} link. 
A link ${\bf L}$ is irreducible if and only if its link exterior is an irreducible 3-manifold.

Let ${\bf L} = (\Sigma, K)$ and ${\bf L'}=(\Sigma',K')$ be links, and choose components $S$ of  $K$ 
and $S'$ of $K$.  Let $M, L \subset \partial N(S)$ and $M', L' \subset \partial N(S')$ be standard 
meridians and longitudes.  Form 
$$\Sigma'' = (\Sigma - {\rm int}\, N(S))\cup (\Sigma'-{\rm int}\,N(S')),$$
pasting along boundaries by matching $M$ to $L'$ and $L$ to $M'$ via a diffeomorphism. 
Then $\Sigma''$ is a smooth homology 3-sphere, and the orientations of  the exteriors of ${\bf L}$ and 
${\bf L'}$ extend to an orientation of $\Sigma''$. 
The link $(\Sigma'', (K-S)\cup (K'-S'))$ is called the {\it splice} of ${\bf L}$ and ${\bf L'}$ along $S$ and $S'$. 

A {\it multilink} is a link ${\bf L} = (\Sigma, S_1\cup \cdots\cup S_n)$ together with an integral multiplicity 
$m_i$ associated with each component $S_i$, with the convention that a component $S_i$ 
with multiplicity $m_i$ means the same thing as $-S_i$ ($S_i$ with reversed orientation) 
with multiplicity $-m_i$. 
We denote the multilink by ${\bf L}(m_1,\ldots,m_n)$. 

A {\it multilink} ${\bf L}(m_1,\ldots,m_n)$ determines an integral cohomology class 
${\bf m}\in H^1(\Sigma - K) = H_1(\Sigma - K)^\ast$ as follows: 
The class ${\bf m}$ evaluated on a 1-cycle $S$ is the linking number 
$${\bf m}(S) = \ell(S,m_1S_1+\cdots + m_nS_n) = \mathop{\sum}_{i=1}^n m_i\ell(S,S_i).$$
In particular, $m_i = {\bf m}(M_i)$ where $M_i$ is a standard oriented meridian of $S_i$. 
We also write ${\bf L}({\bf m})$ for ${\bf L}(m_1,\ldots,m_n)$. 
Each class ${\bf m}\in H^1(\Sigma - K) = H_1(\Sigma - K)^\ast$ 
corresponds to a multilink ${\bf L}({\bf m})$.

Suppose a link ${\bf L} = (\Sigma, K)$ is the result of splicing links ${\bf L}' = (\Sigma', K')$ 
and ${\bf L}'' = (\Sigma'',K'')$ along components $S'$ and $S''$. 
Let $\Sigma_0, \Sigma_0', \Sigma_0''$ be the respective link exteriors so that 
$\Sigma_0 = \Sigma_0' \cup \Sigma_0''$ is joined along $\partial N(S') = \partial N(S'')$. 
Then any cohomology class ${\bf m}\in H^1(\Sigma_0)$ restricts to classes ${\bf m}'\in H^1(\Sigma_0')$ 
and ${\bf m}''\in H^1(\Sigma_0'')$.  We say that ${\bf L}({\bf m})$ is the result of 
{\it splicing} ${\bf L}'({\bf m}')$ and ${\bf L}''({\bf m}'')$. 

A multilink ${\bf L}({\bf m})$ with link exterior $\Sigma_0$ is said to be {\it fibered} 
if there exists a fiber bundle projection $f: \Sigma_0 \to S^1$ such that $f_\ast : H_1(\Sigma_0) \to H_1(S^1)$ 
corresponds to ${\bf m} \in H^1(\Sigma_0) = H_1(\Sigma_0)^\ast$. 

Let $S$ be a component of a link ${\bf L}$ with exterior $\Sigma_0$. 
Then $\partial N(S)$ is said to be {\it compressible} in $\Sigma_0$ if there is a disk $D$ in $\Sigma_0$ 
such that $D\cap \partial N(S) = \partial D$ and $\partial D$ is not contractible in $\partial N(S)$. 

\begin{theorem} 
Let ${\bf L}({\bf m})$ be a multilink which is the splice sum of multilinks 
${\bf L}'({\bf m}')$ and ${\bf L}''({\bf m}'')$ along components $S'$ and $S''$ 
such that $\partial N(S')$ is incompressible in the exterior of ${\bf L}'$ and $\partial N(S'')$ is incompressible in the exterior of ${\bf L}''$. 
Then ${\bf L}({\bf m})$ is fibered if and only if 
${\bf L}'({\bf m}')$ and ${\bf L}''({\bf m}'')$ are fibered. 
\end{theorem}
\begin{proof} 
Let $\Sigma_0, \Sigma_0', \Sigma_0''$ be the respective link exteriors so that 
$\Sigma_0 = \Sigma_0' \cup \Sigma_0''$ is joined along $\partial N(S') = \partial N(S'')$. 
By the Loop Theorem \cite{H}, we have that $\pi_1(\partial N(S'))$ injects into $\pi_1(\Sigma_0')$ 
and $\pi_1(\partial N(S''))$ injects into $\pi_1(\Sigma_0'')$. 
Let $G = \pi_1(\Sigma_0)$, let $A$ be the image of $\pi_1(\Sigma_0')$ in $G$,  
let $B$ be the image of $\pi_1(\Sigma_0'')$ in $G$, and let $C$ be the image 
of $\pi_1(\partial N(S'))$ in $G$.  
Then $\pi_1(\Sigma_0')$ injects onto $A$, and $\pi_1(\Sigma_0'')$ injects onto $B$, and $G= A\ast_C B$ 
by van Kampen's theorem. 
Let $\phi:G\to \integers$ be the homomorphism corresponding to 
${\bf m} \in H^1(\Sigma_0) = H_1(\Sigma_0)^\ast$.
Then $\phi |_A: A \to \integers$ corresponds to ${\bf m}' \in H^1(\Sigma_0') = H_1(\Sigma_0')^\ast$,  
and  $\phi |_B: B \to \integers$ corresponds to ${\bf m}''\in H^1(\Sigma_0'') = H_1(\Sigma_0'')^\ast$. 
As $C\cong \integers^2$, we deduce that ${\rm ker}(\phi |_C)$ is finitely generated and infinite. 
Hence ${\rm ker}(\phi |_A)$, ${\rm ker}(\phi |_B)$, and ${\rm ker}(\phi)$ are infinite.

Suppose that ${\bf L}({\bf m})$ is  fibered with fiber $F$ and fiber bundle projection $f:\Sigma_0\to S^1$.  
Then $F$ is a compact 2-manifold with boundary. 
By Lemma 2.1 of Jaco \cite{J}, the link ${\bf L}$ is irreducible. 
By Lemma 1.1.4 of Waldhausen \cite{W}, the links ${\bf L}'$ and ${\bf L}''$ are irreducible, since 
$\partial N(S')=\partial N(S'')$ is incompressible in $\Sigma_0$. 
The induced homomorphism $f_\ast : G \to \pi_1(S^1)$ corresponds to $\phi:G\to \integers$. 
From the exact sequence for the fibration $f$, 
$$1\to \pi_1(F) \to \pi_1(\Sigma_0) \to \pi_1(S^1) \to \pi_0(F) \to 1,$$
we deduce that ${\rm ker}(\phi)$ is a finitely generated free group and $\phi(G)$ is infinite cyclic. 
Hence ${\rm ker}(\phi |_C)$ is a free group.  The group $C$ is not free, since $C\cong \integers^2$. 
Therefore ${\rm ker}(\phi |_C)$ is not equal to $C$.  
Hence $\phi(C)$ is infinite cyclic. 
Therefore $\phi(A)$  and $\phi(B)$ are infinite cyclic. 
Moreover ${\rm ker}(\phi |_A)$ and ${\rm ker}(\phi |_B)$ are finitely generated by Corollary 2.4.  
Therefore ${\bf L}'({\bf m}')$ and ${\bf L}''({\bf m}'')$ are fibered by Stallings' fibering theorem \cite{S}. 

Conversely, suppose that ${\bf L}'({\bf m}')$ and ${\bf L}''({\bf m}'')$ are fibered. 
Then ${\bf L}$ and ${\bf L}''$ are irreducible. 
Hence ${\bf L}$ is irreducible by Lemma 1.1.4 of \cite{W}, 
since $\partial N(S')=\partial N(S'')$ is incompressible in $\Sigma_0$. 
Moreover ${\rm ker}(\phi |_A)$ and ${\rm ker}(\phi |_B)$ are finitely generated free groups. 
Therefore ${\rm ker}(\phi |_C) \neq C$, since $C$ is not a free group. 
Hence ${\rm ker}(\phi)$ is finitely generated by Corollary 3.1. 
As $\phi(A)$ is infinite cyclic,  $\phi(G)$ is infinite cyclic.  
Thus ${\bf L}({\bf m})$ is fibered by Stallings' fibering theorem.  
\end{proof}

\begin{remark} With the aid of Theorem 4.1, one can prove Theorem 4.2 of \cite{E-N} 
by induction on the number of splice sum components. 
\end{remark}

\vspace{.2in}
\begin{remark} Theorem 4.1 is not true without the hypothesis that $\partial N(S')$ and $\partial N(S'')$ 
are incompressible.  Let $(\Sigma, K)$ be any fibered link.  Choose a closed ball  $B\subset \Sigma-K$ 
and let $S$ be any smooth, oriented, simple, closed curve in $B$. 
Then $(\Sigma, K)$ is the splice sum of $(\Sigma, K\cup S)$ and the unknot $(S^3,S^1)$. 
The link $(\Sigma, K\cup S)$ is not fibered, since it is reducible. 
\end{remark}

\vspace{.2in}
The hypothesis that $\partial N(S')$ and $\partial N(S'')$ are incompressible in Theorem 4.1 is mild 
because of the next theorem. 

\begin{theorem} 
Let ${\bf L}$ be an irreducible link.  
Then {\bf L} has a component  $S$ such that $\partial N(S)$ is compressible in the exterior of ${\bf L}$ 
if and only if ${\bf L}$ is equivalent to the unknot $(S^3,S^1)$. 
\end{theorem}
\begin{proof}
Let $T = \partial N(S)$ and suppose $T$ is compressible in the exterior $\Sigma_0$ of ${\bf L}$. 
Then there is a 2-disk $D$ in $\Sigma_0$ such that $D\cap T = \partial D$ and $\partial D$ 
is not contractible in $T$.  Let $M$ and $L$ be a standard meridian and longitude of $S$ on $T$. 
Then there are relatively prime integers $p$ and $q$ such that 
$\partial D$ is homologous to  $pM+qL$ in $T$.  
The curve $\partial D$ is null homologous in $\Sigma_0$. 
Hence $\ell (\partial D, L) = 0$.  Therefore $p = 0$. 
Hence $q = \pm 1$.  
Therefore there is an annulus in $N(S)$ joining $\partial D$ to $S$, 
and so $S$ bounds a disk in $\Sigma = \Sigma_0\cup N(S)$. 

By Lemma 2.1 of Hempel \cite{H}, we have that $D$ is 2-sided in $\Sigma_0$. 
Hence there is an embedding $h: D\times [-1,1] \to \Sigma_0$ such that $h(x,0) = x$ for all $x\in D$ 
and $h(D\times [-1,1])\cap T = h(\partial D \times [-1,1])$. 
Let $P = h(D\times [-1,1])$.  
Then $\partial P$ is the union of the disk $D_{-1} = h(D\times \{-1\})$, the disk $D_1 = h(D\times \{1\})$, and the annulus $W = h(\partial D \times [-1,1])$ . 
Now $D_{-1}\cup D_1\cup (T^2-{\rm int}\, W)$ is a 2-sphere $F$. 
As $\Sigma_0$ is irreducible, $F$ bounds a closed ball $B$ in $\Sigma_0$. 
We have $\Sigma_0 = B\cup P$ with $B\cap P =  D_{-1}\cup D_1$.  
Hence $\Sigma_0$ is a solid torus, and so $S$ is the only component of ${\bf L}$. 

Now $\Sigma = \Sigma_0\cup N(S)$ is a genus one Heegaard splitting 
of $\Sigma$, and therefore $\Sigma$ is a 3-sphere, since the only homology 3-spheres, with Heegaard genus one, are 3-spheres \cite{H}, pp. 20-21.  
The knot ${\bf L}$ is equivalent to the unknot $(S^3,S^1)$ 
by the Unknotting Theorem \cite{Rolfsen}, p. 103. 
\end{proof}

\section{Cable Knots} 

Let $p$ and $q$ be relatively prime integers, 
and let $(\Sigma, S)$ be a knot. 
We denote by $S(p,q)$ the unique (up to isotopy), oriented, simple, closed curve in $\partial N(S)$ 
that is homologous to $pM+qL$ in $\partial N(S)$. 
The $(p,q)$-{\it cable} of $(\Sigma, S)$ is the knot $(\Sigma, S(p,q))$. 
Let $S_1$ and $S_2$  linked circles in $S^3$. 
For an illustration of $S_1\cup S_2$, see p. 22 of \cite{E-N}. 
By Proposition 1.1 of \cite{E-N}, we have that $(\Sigma, S(p,q))$ is the splice sum 
of $(\Sigma, S)$ and $(S^3, S_1\cup S_2(p,q))$ along $S$ and $S_1$. 

\begin{theorem} 
Let $p$ and $q$ be relatively prime integers such that $q\neq 0$, 
let $(\Sigma,S)$ be a knot, 
and let $(\Sigma, S(p,q))$ be the $(p,q)$-cable of $(\Sigma,S)$. 
Then $(\Sigma, S(p,q))$ is fibered if and only if $(\Sigma, S)$ is fibered. 
\end{theorem}
\begin{proof}
This theorem follows from Theorem 4.2, but 
it is easier to give a direct proof. 
First assume that $p=0$.  Then $q=\pm 1$, and the result is clear, 
since $(\Sigma,S)$ is equivalent to $(\Sigma, S(0,1))$. 
Hence, we may assume that $p\neq 0$.

Choose $N(S(p,q))\subset \Sigma -S$ so that $N(S(p,q))$ intersects $\partial N(S)$ 
in a closed regular neighborhood of $S(p,q)$ in $\partial N(S)$. 
Then $N(S(p,q))\cap \partial N(S)$ is an annulus with core the curve $S(p,q)$. 
Therefore $\partial N(S) - {\rm int}\, N(S(p,q))$ is an annulus in $\partial N(S)$ 
that is parallel to $S(p,q)$. 
Now $N(S) -{\rm int}\, N(S(p,q))$, $N(S)$, and $N(S)\cup N(S(p,q))$ are solid tori 
with core the curve $S$,  
Hence $N(S)\cup N(S(p,q))$ is a closed regular neighborhood of $S$ in $\Sigma$. 
Therefore $E = \Sigma -{\rm int}\,(N(S)\cup N(S(p,q)))$ is homeomorphic to $\Sigma-{\rm int}\,N(S)$. 
Observe that $\Sigma - {\rm int}\,N(S(p,q))$ is the union of 
$E = \Sigma -{\rm int}\,(N(S)\cup N(S(p,q)))$
and the solid torus $V =N(S) -{\rm int}\, N(S(p,q))$ 
intersecting along the annulus $F = \partial N(S) - {\rm int}\, N(S(p,q))$. 

Let $G= \pi_1(\Sigma - {\rm int}\,N(S(p,q))$, let $A$ be the image of $\pi_1(E)$
in $G$, and let $B$ be the image of $\pi_1(V)$ in $G$, 
and let $C$ be the image of $\pi_1(F)$ in $G$.  
As $pq\neq 0$, we have that $\pi_1(F)$ injects into both  $\pi_1(E)$ and $\pi_1(F)$. 
Then $\pi_1(E)$ injects onto $A$, and $\pi_1(F)$ injects onto $B$, 
and $G=A\ast_C B$ by van Kampen's theorem. 
The group $C$ is infinite cyclic with generator represented by the core curve of $F$. 
This curve is homologous to $pM$ in $\Sigma - {\rm int}\,N(S)$ and to $qL$ in $V$.  
From the Mayer-Vietoris exact sequence
$$H_1(F) \to H_1(E)\oplus H_1(V) \to H_1(\Sigma-{\rm int}\,N(S(p,q))) \to 0,$$
we deduce that 
$G_{ab} \cong \integers\oplus\integers/\langle (p,-q)\rangle \cong \integers.$
By the same argument as in the proof of Theorem 3.2, we deduce that 
$[G,G]$ is a free group of finite rank if and only if $[A,A]$ is a free group of finite rank. 

The solid torus $V$ is irreducible. 
The annulus $F$ is incompressible in $E$ and in $V$ by Lemma 1.1.3 of \cite{W}.  
By Lemma 1.1.4  of  \cite{W}, we have that $\Sigma-{\rm int}\,N(S(p,q))$ is irreducible 
if and only $E$ is irreducible if and only if $\Sigma-{\rm  int}\,N(S)$ is irreducible. 
A compact 3-manifold fibered over $S^1$, with fiber not a 2-sphere, is irreducible by Lemma  2.1 of \cite{J}. 
Therefore $(\Sigma, S(p,q))$ is fibered if and only if $(\Sigma, S)$ is fibered 
by Stallings' fibering theorem \cite{S}. 
\end{proof}

\begin{remark} The ``\,if " part of Theorem 5.1 was first proved by J. Simon \cite{Simon} 
by a more direct argument using the splice sum decomposition of $(\Sigma, S(p,q))$. 
In \cite{Simon}, Simon remarked that this result could also be proved using 
Stallings' fibering theorem. 
\end{remark}

\vspace{.2in}
\begin{example} Let $(S^3,K)$ be the trefoil knot. 
Then $(S^3,K)$ is the $(2,3)$-cable of the unknot $(S^3,S^1)$. 
Hence $(S^3,K)$ is fibered by Theorem 5.1. 
Let $(\Sigma, S)$ be the knot obtained from the trefoil knot 
by Dehn surgery with coefficient $-1/2$. 
Then $\Sigma$ is the Brieskorn homology 3-sphere $\Sigma(2,3,11)$  
according to R. Fintushel and R. Stern \cite{F-S}. 
We have that $\Sigma-{\rm int}\,N(S) = S^3-{\rm int}\,N(K)$, 
and so $(\Sigma, S)$ is fibered. 
Hence $(\Sigma, S(1,2))$ is fibered  by Theorem 5.1. 
In \cite{R}, it is proved that 
$$\pi_1(\Sigma-S(1,2)) \cong \langle x,y; x^2y^2x^2y^{-1}\rangle,$$ 
and  $\langle x,y; x^2y^2x^2y^{-1}\rangle$ is not isomorphic to the group 
of any knot in $S^3$. 


\end{example}

\end{document}